\documentclass[english]{ourlematema}

\title{On realizability of lines on tropical cubic surfaces and the Brundu-Logar normal form}
\titlemark{Relative realizability and the Brundu-Logar normal form}


\author{Alheydis Geiger}
\address{Department of Mathematics, University of T\"{u}bingen, 
Germany\\
	\email{alheydis.geiger@math.uni-tuebingen.de}}

\usepackage{graphicx}
\usepackage{amsmath}								
\usepackage{amssymb}
\usepackage{amsthm}
\usepackage{amscd}
\usepackage[square,numbers,sort&compress,sectionbib]{natbib}
\usepackage{amsfonts}
\usepackage{stmaryrd}
\usepackage[all]{xy}
\usepackage{caption}
\usepackage{float}
\usepackage{subcaption}
\usepackage{euler}

\usepackage{extarrows}
\usepackage[colorlinks, linktocpage, citecolor = purple, linkcolor = blue]{hyperref}
\usepackage{color}
\usepackage{tabularx}
\usepackage{ltxtable}
\usepackage{longtable}
\usepackage{multirow}

\usepackage{tikz}									
\usetikzlibrary{matrix}
\usetikzlibrary{patterns}
\usetikzlibrary{matrix}
\usetikzlibrary{positioning}
\usetikzlibrary{decorations.pathmorphing}
\usetikzlibrary{cd}
\usepackage[shortlabels]{enumitem}
\usepackage{color}
\usepackage{caption}
\usepackage{natbib}

\newcommand{\conv}[1]{\text{conv}(#1)}
\newcommand{\trop}[1]{\text{trop}(#1)}

\newcommand{\val}{\text{val}}

\theoremstyle{plain}
\newtheorem*{theoremno}{Theorem}

\definecolor{light-gray}{gray}{0.7}
\begin{document}
	\bibliographystyle{lematema}
	\maketitle
	
	\begin{abstract}
		We present results on the relative realizability of infinite families of lines on general smooth tropical cubic surfaces. Inspired by the problem of relative realizability of lines on surfaces, we investigate the information we can derive tropically from the Brundu-Logar normal form of smooth cubic surfaces.
		In particular, we prove that for a residue field of characteristic $\neq 2$ the tropicalization of the Brundu-Logar normal form is not smooth. We also take first steps in investigating the behavior of the tropicalized lines. 
	\end{abstract}
	
	\section{Introduction}
	Since 1849 it has been a well-known fact that every classical smooth cubic surface contains exactly $27$ lines \cite{Ca,Sal}. With the rise of tropical geometry the question about an analogous result about tropical cubic surfaces was inevitable. In 2007 Vigeland proved in \cite{Vig07} that on smooth tropical cubic surfaces, there can be infinitely many tropical lines. Accordingly, the following questions arose:
	\begin{que}[Relative Realizability]
		For a given pair $(L,X)$ of a tropical line $L$ on a tropical cubic surface $X$, is there a pair $(\ell,C)$ of a classical line $\ell$ on a classical cubic surface $C$ such that $\trop{\ell}=L$ and $\trop{C}=X$?\end{que}
		\begin{que}[Relative Realizability, Lifting Multiplicities]
		For a pair $(L,X)$ and a fixed lift $C$ of $X$, how many lines $\ell\subset C$ tropicalize to $L$, i.e. what is the \textit{lifting multiplicity} of $L$ with respect to $C$?\end{que}
	The question of the geometric behavior of the tropical lines and the problem of their relative realizability	has been worked on for more than a decade and is not yet completely solved \cite{Vig07,Vig10,PaVi19,BS15,BK12,Ge18,JPS19,CuDe19,PSS19}.\\
	With respect to the question of lifting multiplicities, there has been a recent new development: in \cite{PSS19} an octanomial model for cubic surfaces is presented, which satisfies that all $27$ lines on a tropically smooth cubic in this form have distinct tropicalizations \cite[Theorem 3.4]{PSS19}, i.e. all lifting multiplicities are one or zero. 
	Also, the conjecture is posed that the $27$ lines on a tropically smooth cubic surface have distinct tropicalizations \cite[Conjecture 4.1]{PSS19}. Moreover, according to \cite{PSS19}, Kristin Shaw has announced a proof that every tropically smooth family of complex cubic surfaces contains 27 lines with distinct tropicalizations.
	
By tools introduced in \cite{Vig10,PaVi19} it is possible to divide tropical lines into isolated lines and infinite families. On general smooth tropical cubics there can only occur two types of families; see Proposition \ref{prop:2types} and \cite{PaVi19,Ge18}. The lifting behavior of one type of family  over characteristic $0$ has been investigated in \cite{BS15,BK12}.
Our first main result, Theorem \ref{theorem:relreal}, is a generalization of this result to fields of arbitrary characteristics, obtained by complementary methods. Also, we present two trend-setting new examples, one for the second family type in Example \ref{ex:3J} and one on the unsolved case to Theorem \ref{theorem:relreal} in Example \ref{ex:3I-1}.

	Unfortunately, the computational expenditure to compute the $27$ lines on a generic classical smooth surface is very high, impeding an easy access to examples of more complex lifts over arbitrary characteristic. The existence of a normal form of classical cubic surfaces which allows the direct computation of all 27 lines from the parameters of the cubic polynomial by simple formulas seems to present a new perspective to solving this problem. Therefore, we ask
	\begin{que}[{\cite[Question 26]{27questions}}]
		How to compute the Brundu-Logar normal form in practice? What does it tell us tropically?
	\end{que}

	The first part of this question has been computationally answered for generic cubics over the $p$-adic fields by Avinash Kulkarni \cite{Ku19}. 
	As a smooth cubic is transformed in the Brundu-Logar normal form by projective linear transformations, this question is closely related to 
	\begin{que}[{\cite[Question 12]{27questions}}]
		How can we decide if a given polynomial defines a smooth tropical surface after a linear transformation of $\mathbb{P}^3$?
	\end{que}
	Unfortunately, the Brundu-Logar normal form of a lift of any chosen smooth tropical cubic is itself no longer tropically smooth. More precisely, our second main result is the following theorem:
	\begin{theoremno}[Theorem \ref{theorem:val(2)=0}]
		If $\val(2)=0$, the tropicalization of any cubic in Brundu-Logar normal form is not tropically smooth.
	\end{theoremno}

However, we can still investigate the tropicalization of cubics in Brundu-Logar normal form and look into the behavior of the tropicalizations of the 27 lines. 
We will see in Section \ref{sec:BLNF_lines} that the statement of Theorem 3.4 from \cite{PSS19} does not hold for the Brundu-Logar normal form, as we can have higher lifting multiplicities in this setting, as illustrated in Examples \ref{ex:trivsub} and \ref{ex:1}. These examples prove that \cite[Conjecture 4.1]{PSS19} does not hold for non-smooth tropical cubics.

This paper is organized as follows. In Section \ref{sec:relreal} we present results on the relative realizability of lines on tropical cubic surfaces. In Section \ref{sec:BLNFandtrop} we study the tropical Brundu-Logar normal form and its tropicalized lines. Section \ref{sec:ex} presents significant examples of realizable lines in tropical cubic surfaces, one of them obtained by transforming to the Brundu-Logar normal form.
\\
\
\\
\noindent\textbf{Acknowledgments.} I am very grateful to Hannah Markwig for her continued support and advice during this project. I warmly thank Diane Maclagan for her support during my research at the University of Warwick. Also I want to thank Marta Panizzut, Michael Joswig and Bernd Sturmfels for insightful and inspiring conversations on cubic surfaces. Very special thanks to Marta Panizzut and Avinash Kulkarni for their ready and generous assistance with the computations. \\
I also want to thank Sara Lamboglia and the anonymous referee for their helpful comments on the presentation of this paper.

	\section{Relative realizability of infinite families of lines}\label{sec:relreal}
	Let $3\Delta_3:=\conv{(0,0,0),(3,0,0),(0,3,0),(0,0,3)}.$ Denote the facets of $3\Delta_3$ with $F_i, F_j,F_k,F_l.$
	Using the concept of decorations introduced in \cite{Vig10}, we can classify lines on surfaces by their way of being contained in the surface: passing through vertices or edges. Decorations allow one to distinguish isolated lines and infinite families of lines contained in the tropical surface. This concept was refined by the theory of motifs introduced in \cite{PaVi19}, where also a thorough classification of all motifs of lines on general smooth surfaces of varied degree can be found, completing the started classification from \cite{Vig10}. The concept of generality used in this section is the one introduced by \cite{Vig10}.

\begin{dfn}[{\cite[Definition 6]{PaVi19}}]
    A motif of a tropical line $L$ on a tropical surface $X$ is a pair $(G,\mathcal{R})$, where the primal motif $G$ is a decoration of the underlying graph of $L$ with a finite number of dots and vertical line segments. Every dot stands for a vertex of $X$ contained in that cell of $L$ corresponding to the decorated edge or vertex of the graph, while a vertical line segment stands for an edge of $X$ that $L$ intersects in its relative interior with the cell whose corresponding part in the graph is decorated.
    The dual motif $\mathcal{R}$ is the subcomplex in the dual subdivision of $X$ dual to the union of all cells of $X$ that $L$ passes through: $\mathcal{R}=\bigcup_{p\in L} (\min\{C\subset X \text{ cell of } X \,\,  |\,\, p\in C \})^{\vee}$.
\end{dfn}{}
	\begin{prop}[{\cite[Proposition 23]{PaVi19}}, {\cite[Section 3.2]{Ge18}}]\label{prop:2types} Only the motifs 3I and 3J allow infinite families on general smooth tropical cubic surfaces; see Figure \ref{fig:3Iand3J}.
	\end{prop}
	\begin{figure}
		\centering
		\begin{subfigure}{0.2\textwidth}
		\centering
		  \includegraphics[width = 0.7\linewidth]{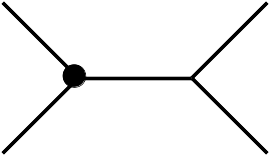}
		  \caption{Motif 3I}
		\end{subfigure}\qquad
			\begin{subfigure}{0.2\textwidth}
			\centering
		  \includegraphics[width = 0.7\linewidth]{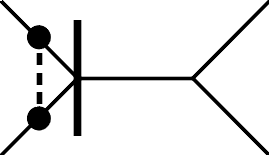}
		  \caption{Motif 3J}
		\end{subfigure}
		\caption{Infinite families on general smooth tropical cubic surfaces}
		\label{fig:3Iand3J}
	\end{figure}

	\begin{rem} The dual motif to 3I is a unimodular tetrahedron, which has one edge in $F_i\cap F_j$ and at least a one-dimensional intersection with both $F_k$ and $F_l$. This is called a \textit{$4$-exit tetrahedron}. 
		The property of relative realizability of families of motif 3I depends on how the dual motif, a 4-exit tetrahedron, is contained in the subdivision of the Newton polytope $3\Delta_3$ of the cubic surface. 
		For smooth cubic surfaces, there are two ways up to coordinate changes to include a 4-exit tetrahedron in $3\Delta_3$. 
	\end{rem}

	\begin{dfn}
	A family of lines $(L_a)_{a\geq 0}$ on a smooth cubic surface is \textit{of type 3I-1} if it has motif 3I with the dual motif contained in $3\Delta_3$ in a position equivalent to $\{(0,0,0),(0,0,1),(2,1,0),(1,0,2)\}$ up to coordinate changes; see Figure \ref{fig:3I-1}. 
		 It is \textit{of type 3I-2} if it is of motif 3I and the dual motif is contained in $3\Delta_3$ in a position equivalent by coordinate changes to $\{(0,0,0),(0,0,1),(2,1,0),(1,1,1)\}$; see Figure \ref{fig:3I-2}. 
	\end{dfn}
	\begin{figure}
		\centering
			\begin{subfigure}{0.4\textwidth}
			\centering
		  \includegraphics[width = 0.9\linewidth]{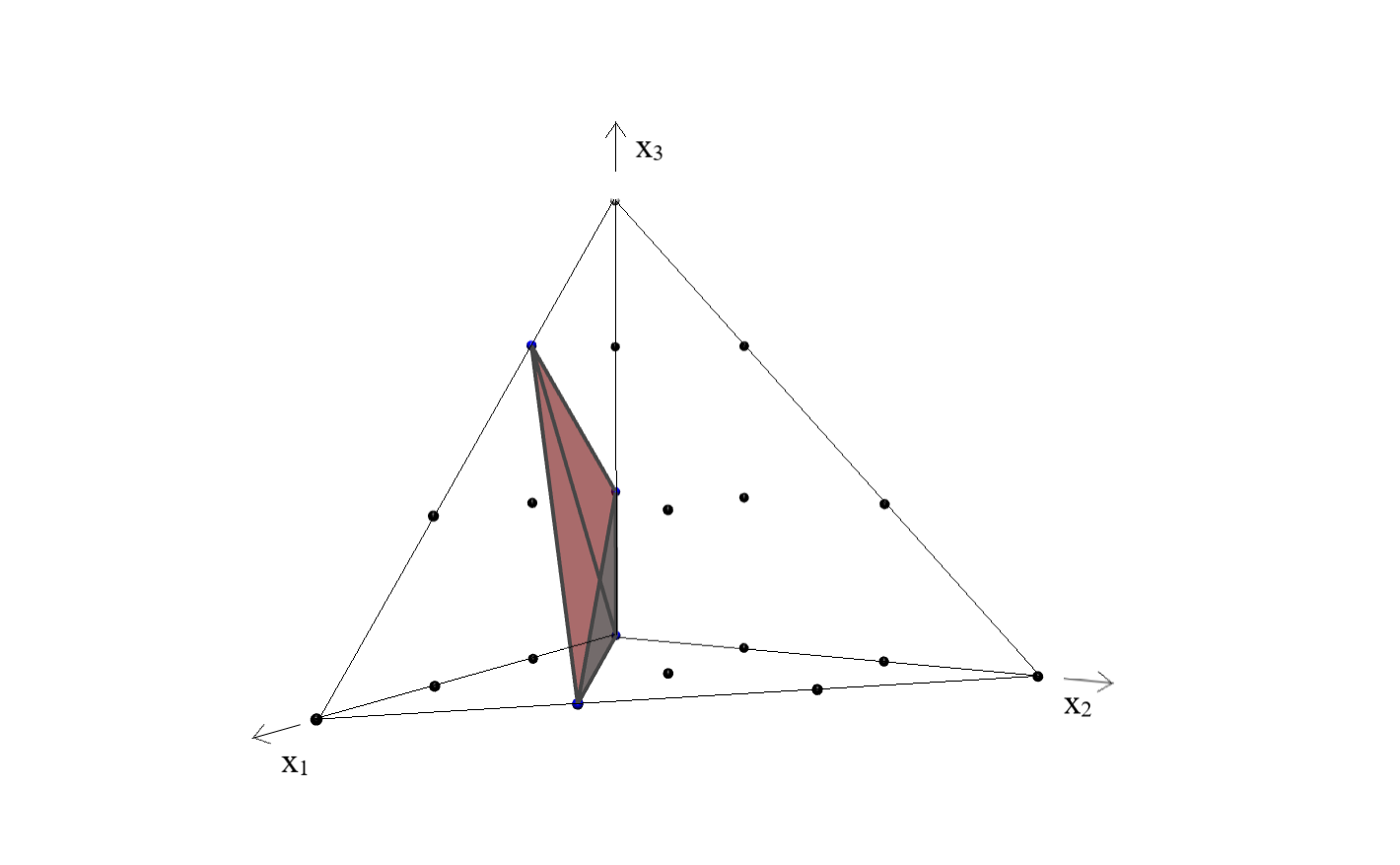}
		  \caption{Type 3I-1}\label{fig:3I-1}
		\end{subfigure}
		\begin{subfigure}{0.4\textwidth}
		\centering
		  \includegraphics[width = 0.9\linewidth]{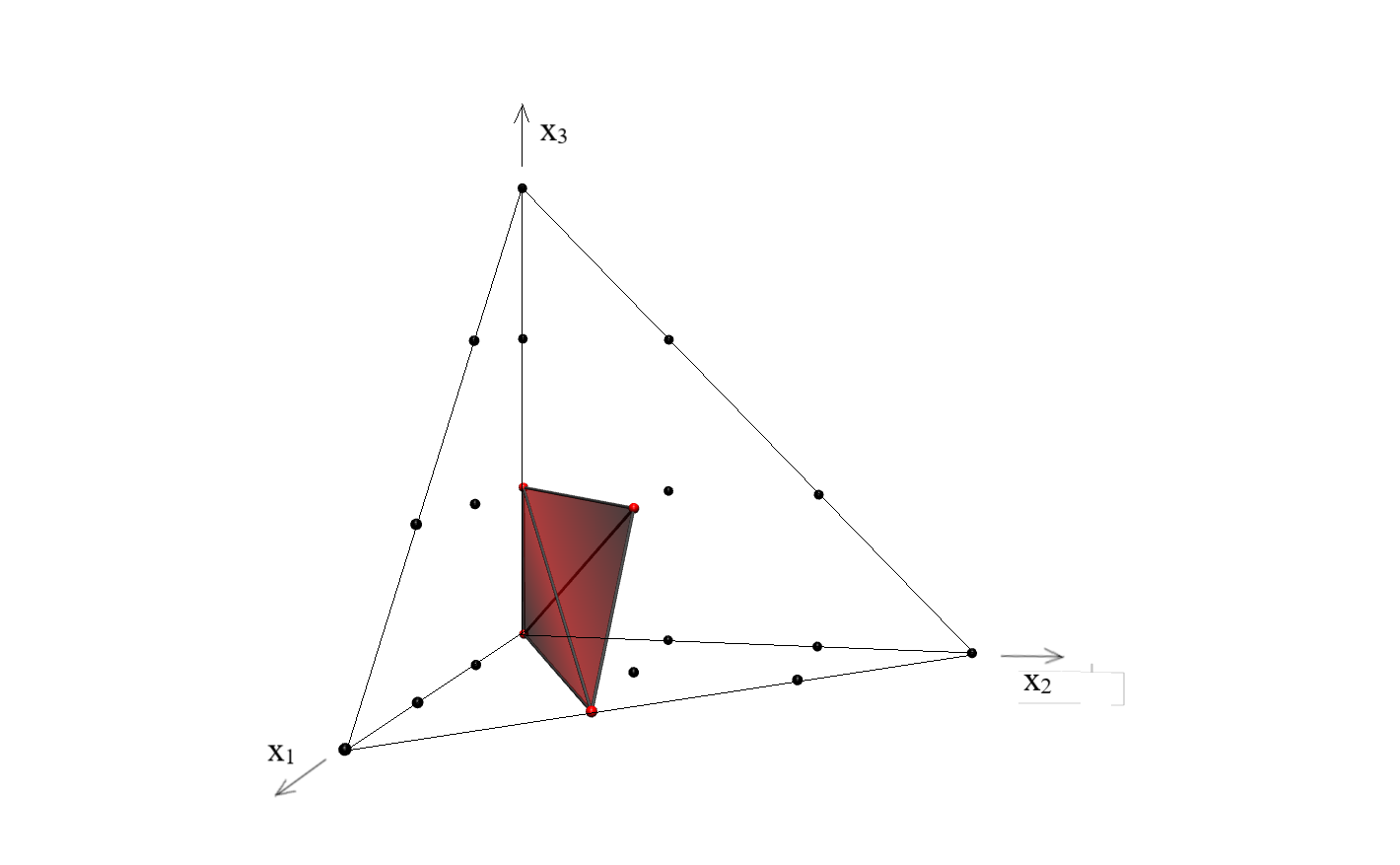}
		  \caption{Type 3I-2}\label{fig:3I-2}
		\end{subfigure}
		\caption{Position of the 4-exit tetrahedron in both types}
		\label{fig:3I-1and3I-2}
	\end{figure}	
		A tropical line is called \textit{degenerate,} if the bounded middle line segment has length zero.
	\begin{prop}\label{cor:relreal}Over $K=\mathbb{C}$ a line $L$ of motif 3I on a tropical cubic surface $X$ lifts if and only if it is degenerate and of type 3I-1.
	\end{prop}
	\begin{proof} 
	This is a corollary of Theorem 7.2 in \cite{BS15} which states, that for an algebraic surface $\mathcal{S}$ over $(\mathbb{C}^*)^3$ with Newton polytope a unimodular 4-exit tetrahedron, the tropical line $\text{star}_{(0,0,0)}(L_l)\subset\mathcal{S}$ is approximable by a complex algebraic line $\mathcal{L}_l\subset\mathcal{S}$ if and only if $l=0$ and the Newton polytope is of the form $\{(0,0,0),(0,0,1),(2,1,0),(1,0,2)\}$. We can assume the vertex of $X$ contained in $L$ to be $(0,0,0)$ and apply this	locally to  $\text{star}_{(0,0,0)}(X)$.
	\end{proof}
	\begin{thm}[{\cite[Theorems 4.1.9 and 4.1.16]{Ge18}}]\label{theorem:relreal}
		Let $(K,\val)$ be an algebraically closed field with valuation and let $\Bbbk$ be its residue field. Let $X$ be a general smooth tropical cubic surface containing a family of lines $(L_a)_{a > 0}$.\\
		If the family $(L_a)_{a > 0}$ is of type 3I-1 with $\text{char}(\Bbbk)\neq 2$
		or the family is of type 3I-2, then for any $f\in K[x_0,\dots,x_3]$ homogeneous with $X=V(\trop{f})$ the family cannot be realized as lines on $V(f)$, i.e. the lifting multiplicity of any 
		non-degenerate member of the family is zero.
	\end{thm}
	\begin{proof} Since we consider our cubic $f\in K[x_0,\dots,x_3]$ homogeneous, we consider everything over $\mathbb{R}^4/\mathbb{R}\mathbf{1}$.
		Without loss of generality we can assume that the monomials corresponding to the vertices of the dual motif of the family have coefficients of valuation zero, while the coefficients of all other monomials in $f$ have strictly positive valuation. We write $\mathbf{0}:=(0,0,0,0).$\\
		For $a>0$ let $I_a\subset K[x_0,...,x_3]$ be a linear homogeneous ideal, such that $\trop{V(I_a)}=L_a$. If the lines are realizable, we have $0\neq\text{in}_{\mathbf{0}}(f)\in \text{in}_{\mathbf{0}}(I_a).$ 
	 The initial ideal of $I_a$ in $\mathbf{0}$ is given as $\text{in}_{\mathbf{0}}(I_a) = \langle x_0-x_1+x_3, x_1+x_2 \rangle$. This basis is a Gr\"obner basis with the lexicographic order with $x_3\succ \ldots \succ x_0.$\\
		If $(L_a)_{a > 0}$ has type 3I-1, we obtain $ \sigma:=\text{supp}(\text{in}_{\mathbf{0}}(f)) =\{x_0^3, x_0^2x_3, x_1x_3^2,x_1^2x_2 \}.$  By Gr\"obner basis theory, we know that there exits a polynomial $\alpha x_0^3 + \beta x_0^2x_3 + \gamma x_1x_3^2 + \delta x_1^2x_2\in\text{in}_{\mathbf{0}}(I_a)$ of support $\sigma$ if and only if its polynomial division with the Gr\"obner basis of $\text{in}_{\mathbf{0}}(I_a)$ has remainder zero. Carrying out the polynomial division, we find the remainder
	$$(\alpha-\beta) x_0^3 + (\beta+\gamma)x_0^2x_1 -2\gamma x_0x_1^2 + (\gamma -\delta)x_1^3 \in \Bbbk[x_0,...,x_3].$$
		If $\text{char}(\Bbbk)\neq 2$, this polynomial is zero if and only if $\alpha = \beta=\gamma =\delta =0$.\\
		For $(L_a)_{a > 0}$ of type 3I-2, $\text{supp}(\text{in}_{\mathbf{0}}(f)) =\{x_0^3, x_0^2x_3, x_1^2x_2, x_1x_2x_3 \}.$ The remainder of $\alpha x_0^3 + \beta x_0^2x_3 + \gamma x_1^2x_2 + \delta x_1x_2x_3$ by polynomial division is $$(\alpha-\beta) x_0^3 + \beta x_0^2x_1 -\delta x_0x_1^2 + (\gamma +\delta)x_1^3 \in \Bbbk[x_0,...,x_3].$$
		This polynomial is zero if and only if $\alpha = \beta=\gamma =\delta =0$. 
	\end{proof}
	Theorem \ref{theorem:relreal} deals with the non-degenerate lines of motif 3I. We can use the same techniques to prove the following theorem. 
	\begin{thm}\label{theorem:deg3I}
		Let $(K,\val)$ be an algebraically closed field with valuation and let $\Bbbk$ be its residue field. Let
		$X$ be a general smooth tropical cubic surface containing a family of lines of type 3I-2.
		The degenerate line of this family cannot be realized on any lift $V(f)$ of $X$, $f\in K[x_0,\dots,x_3]$ homogeneous.
	\end{thm}
	\begin{proof}
	As before we consider everything over $\mathbb{R}^4/\mathbb{R}\mathbf{1}$ and assume that the monomials corresponding to the vertices of the dual motif of the family have coefficients of valuation zero, while all other coefficients in $f$ have strictly positive valuation. Again we denote $\mathbf{0}:=(0,0,0,0).$\\
	Let $L$ be the degenerate line of the family of type 3I-2 on $X$ and let $I\subset K[x_0,...,x_3]$ be a linear homogeneous ideal, such that $\trop{V(I)}=L$.\\
	We use the same argument as in the proof of Theorem \ref{theorem:relreal}:	if the line is realizable, it follows that $0\neq\text{in}_{\mathbf{0}}(f)\in \text{in}_{\mathbf{0}}(I).$ 
	The initial ideal of $I$ in $\mathbf{0}$ can only be determined up to a $\lambda\neq 0,-1$: $\text{in}_{\mathbf{0}}(I) = \langle x_0+\lambda x_1-x_3, x_0-x_1-x_2 \rangle$, given by a Gr\"obner basis with the lexicographic order with $x_3\succ \ldots \succ x_0.$
		If $L$ is of type 3I-2, we know $\sigma:= \text{supp}(\text{in}_{\mathbf{0}}(f)) =\{x_0^3, x_0^2x_3, x_1^2x_2,x_1x_2x_3 \}.$ 
		By polynomial division of a polynomial $\alpha x_0^3+\beta x_0^2x_3 +\gamma x_1^2x_2+\delta x_1x_2x_3$ of support $\sigma$ with the Gr\"obner basis, we obtain the remainder:
		$$(-\gamma-\lambda\delta) x_1^3 + (\lambda\delta+\gamma-\delta) x_1^2x_0 +(\lambda\beta +\delta) x_1x_0^2 + (\alpha+\beta)x_0^3 \in \Bbbk[x_0,...,x_3].$$
		Since $\lambda\neq0,-1$, this polynomial is zero if and only if $\alpha = \ldots =\delta =0$.
	\end{proof}
	\begin{rem}\label{rem:3I-1}
	The same techniques do not reveal similar results for type 3I-1. Indeed the lifting behavior of the two types is quite different.

	Example \ref{ex:3I-1} shows that an analogous statement to Theorem \ref{theorem:deg3I} for the degenerate lines of type 3I-1 over characteristic $p$ is in general not true, suggesting an extension of the validity of Proposition \ref{cor:relreal}.
	
Also recall that for families of type 3I-1 in $\text{char}(\Bbbk)=2,$ there is a non-zero solution for the coefficients $\alpha,\,\beta,\,\gamma,\,\delta $ in the proof of Theorem \ref{theorem:relreal}. Example \ref{ex:3I-1} shows a surface on which a non-degenerate representative of a family of type 3I-1 lifts, giving rise to the following question.
	\end{rem}
	\begin{que}\label{que:3I-1}	For any member $L$ of a family of type 3I-1 on a smooth tropical surface $X$, is there a lift of $X$  defined over a field $K$ with residue field $\Bbbk$ of $\text{char}(\Bbbk)=2$ containing a lift of $L$?
	\end{que}

It remains to investigate families of motif 3J. Similar to lines of type 3I-1, we can find a lift of a smooth tropical cubic surface defined over $\mathbb{Q}_p$, for $p=2$ and $p=5$ each, where a non-degenerate member of the family of motif 3J is realizable. The details are shown in Example \ref{ex:3J}.
The arising question is

		\begin{que}\label{que:3J}
		For any member $L$ of a family of motif 3J on a smooth tropical surface $X$, is there a lift of $X$ over any algebraically closed field $K$ containing a lift of $L$?
		\end{que}

	\section{The Brundu-Logar normal form and its tropicalization}\label{sec:BLNFandtrop}
The fact that the Pl\"ucker coordinates of the lines on cubics in Brundu-Logar normal form are known explicitly motivates our study of their tropicalizations. In this section let $K$ be algebraically closed, $\text{char}(K)=0.$

	\subsection{The classical Brundu-Logar normal form}\label{sec:BLNF}

	\begin{dfn}\label{def:BLNF}
		A homogeneous cubic polynomial $F\in K[x_1,...,x_4]$ is in \textit{Brundu-Logar normal form} if 
		\begin{align*}
		F=\,&2ax_1^2x_2+(b-g)x_1^2x_3+(-2a)x_1x_2^2+(d+g)x_1x_2x_3+(a-c-d)x_1x_3^2 + \nonumber\\&(b+g)x_1x_2x_4 +(-a-b-c)x_1x_3x_4+(d-g)x_2^2x_4+(a+c-d)x_2x_3x_4\nonumber\\&+(-a-b+c)x_2x_4^2,
		\end{align*}
		where $(a,b,c,d,g)\notin\Sigma$ and $\Sigma:=V(\sigma)\subset \mathbb{P}^4$ is the hypersurface defined by
			\begin{align*}
			\sigma:=\, &c(a+b-c)(2a+b-d)(a-c-d)(a+c+g)(a+c-g)\\&
			(4ac-g^2)(a^2+ac-2ad+ag+d^2-dg)\\&(a^2+2ab+ac-ag+b^2-bg)\\&
			(4a^2+3ab-4ac-3ad-bc-2bd+bg+cd+dg)\\&(4a^3+4a^2b+8a^2c-4a^2d+ab^2-4abc-2abd\\&+2abg+4ac^2+4acd+ad^2+ 2adg+b^2c+b^2g+2bcd\\&-2bcg+cd^2-2cdg-d^2g).
			\end{align*}\end{dfn}
	
	\begin{thm}[{\cite[Theorem 1.4 and 1.7]{BL98}}]
	Every smooth homogeneous cubic polynomial $f\in K[x_1,x_2,x_3,x_4]$  can be brought into Brundu-Logar normal form by applying projective transformations.\\
	Every cubic in Brundu-Logar normal form is smooth and thus contains exactly $27$ lines. 
	In particular, the points in $\mathbb{P}^4\setminus\Sigma$ parametrize all smooth cubics.
	\end{thm}

The question how to obtain the Brundu-Logar normal form for a given cubic is  
answered by a Magma script written by Avinash Kulkarni for generic cubics over $p$-adic fields \cite{Ku19}.

	\subsection{Tropicalizing the Brundu-Logar normal form}
	Let $P:=(a,b,c,d,g)\in\mathbb{P}^4\setminus \Sigma$ and $f_P$ the cubic in Brundu-Logar normal form defined by $P$. We denote by $S_P$ the corresponding cubic surface. The tropicalization of $f_P$ in Brundu-Logar normal form is:
	\begin{align}
	\trop{f_P}=\min\{&\val(2a)+2x_1+x_2,\val(b-g)+2x_1+x_3,\val(-2a)+x_1+2x_2,\nonumber\\&\val(d+g)+x_1+x_2+x_3,\val(a-c-d)+x_1+2x_3 ,\nonumber\\&\val(b+g)+x_1+x_2+x_4, \val(-a-b-c)+x_1+x_3+x_4,\nonumber\\&\val(d-g)+2x_2+x_4,\val(a+c-d)+x_2+x_3+x_4,\nonumber\\&\val(-a-b+c)+x_2+2x_4\}.\nonumber
	\end{align}
	If all parameters of $f_P$ are not zero, the Newton polytope is given by
	\begin{align}
	\text{Newt}(f_P)=\conv{&(0,1,0,2),(1,1,0,1),(1,0,1,1),(0,2,0,1),(0,1,1,1),\nonumber\\&(2,1,0,0),(2,0,1,0),(1,2,0,0),(1,1,1,0),(1,0,2,0)}.\nonumber
	\end{align}
	\begin{figure}
		\centering
		\begin{minipage}{0.55\textwidth}
		\centering
		  \includegraphics[width=1\linewidth]{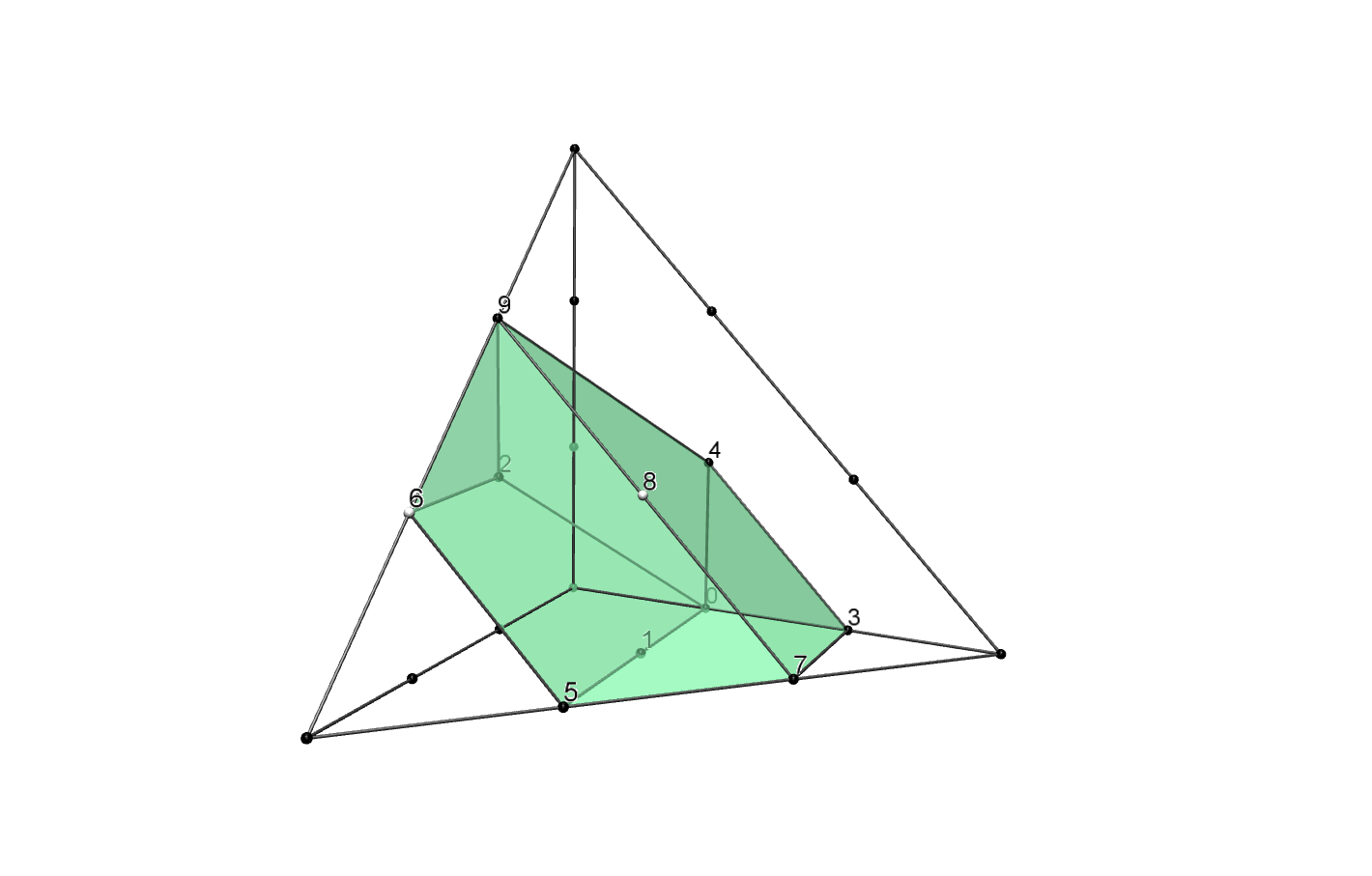}
		  \caption{Newton polytope of the tropicalized Brundu-Logar normal form}\label{fig:Newt(f)}
		\end{minipage}\quad
		\begin{minipage}{0.4\textwidth}
		\centering
		  \footnotesize	\begin{align*}
	\begin{matrix}
	0  &(0,1,0,2) & \val(-a-b+c)\\
	1  &(1,1,0,1) & \val(b+g)\\
	2  &(1,0,1,1) & \val(-a-b-c)\\
	3  &(0,2,0,1) & \val(d-g)\\
	4  &(0,1,1,1) & \val(a+c-d)\\
	5  &(2,1,0,0) & \val(2a)\\
	6  &(2,0,1,0) & \val(b-g)\\
	7  &(1,2,0,0) & \val(-2a)\\
	8  &(1,1,1,0) & \val(d+g)\\
	9  &(1,0,2,0)& \val(a-c-d)\\
	\end{matrix}
	\end{align*}
	\normalsize
	\captionof{table}{Weights and vertices for a cubic in Brundu-Logar normal form}\label{fig:tabweightsandvertices}
		\end{minipage}{}
	\end{figure}
Dehomogenized with respect to $x_4$ and embedded in the tetrahedron $3\Delta_3$ this polytope is shown in Figure \ref{fig:Newt(f)}.
	
	The weight vector induced by a cubic in Brundu-Logar normal form with $P=(a,b,c,d,g)\in\mathbb{P}^4\setminus\Sigma$ is
	\begin{align}
	\omega_P=(&\val(-a-b+c),\val(b+g),\val(-a-b-c),\val(d-g),\val(a+c-d),\nonumber\\&\val(2a),\val(b-g),\val(-2a),\val(d+g),\val(a-c-d)).\nonumber
	\end{align}

\begin{rem}
	By definition, a tropical variety $V(\trop{f})$ is smooth if and only if the dual subdivision of $\text{Newt}(f)$ is unimodular \cite[Section 4.5]{MS15}.
	For all coefficients of $P\in\mathbb{P}^4\setminus\Sigma$ not zero, we have $\text{vol}(\text{Newt}(f_P))=\frac{5}{3}$,
	so a unimodular triangulation contains exactly $10$ tetrahedra. 
\end{rem}
	By \cite[Theorem 1.4]{BL98}, every classical Brundu-Logar normal form is smooth. However, the conditions on the weights to a cubic in Brundu-Logar normal form do not allow unimodular triangulations.
	
	\begin{thm}\label{theorem:val(2)=0}
		If $\val(2)=0$, the tropicalization of any cubic in Brundu-Logar normal form is not tropically smooth.
	\end{thm}	
	In particular, this statement holds for most choices of fields with a valuation $(K,\val).$
	\begin{proof}
		The idea is to make a case distinction of which combinations of the parameters in $P=(a,b,c,d,g)\in\mathbb{P}^4\setminus\Sigma$ have minimal valuation. In every case, we obtain either that too many vertices have minimal valuation, that $4$ vertices contained in a plane have minimal valuation or that the $4$ vertices of minimal valuation form a tetrahedron of size larger than $\frac{1}{6}$. These arguments are valid independent of $P$ containing zeros or coordinates of $\omega_P$ being $\infty$ and thus of the exact volume of $\text{Newt}(f_P)$. In the proof we classify all possible cases. For each case we find obstructions to unimodular triangulations. 
		For brevity, we abstain from listing all the cases; rather we show the concept by an exemplary case:\\
		Let $\val(a)=\val(c)<\val(b),\val(d),\val(g)$. \\ For the numeration of the vertices see Table \ref{fig:tabweightsandvertices}.
		Since on the level of minimal valuations we can only have cancellations in either $a+c$ or $a-c$, at least one of $\val(-a-b+c)$ and $\val(-a-b-c)$ has value $\val(a)$, similar with $\val(a-d+c)$ and $\val(a-d-c)$. If $a$ and $c$ have no cancellation, we obtain that the six vertices 0, 2, 4, 5, 7, 9 have the same minimal valuation. But a subdivision of a polytope in $\mathbb{R}^3$ can only be unimodular if a maximum of $4$ vertices have minimal valuation. We need to have cancellations:
		\begin{enumerate}
			\item Cancellation in $a-c$.\\
			It follows that $\val(a)=\val(c)=\val(a+c).$
			Therefore, the vertices with numbers 2, 4, 5, 7 have same minimal valuation. However, these $4$ vertices are contained in a plane and therefore induce subpolytopes in the subdivision which are not unimodular. 
			\item Cancellation in $a+c$.\\ We obtain $\val(a)=\val(c)=\val(a-c).$ So the vertices with numbers 0, 5, 7, 9 form a polytope in the triangulation. However, this tetrahedron is too big, since the edge from vertex number 0 to 5 contains another lattice point in its interior.
		\end{enumerate}
	\end{proof}	
 
In contrast to the octanomial model presented in \cite{PSS19}, the tropicalized Brundu-Logar normal form is not smooth for most choices of valuated fields. Additionally in Example \ref{ex:1} we will see that an analogous statement to \cite[Theorem 3.4]{PSS19} on the bound $1$ for lifting multiplicities does not hold for the tropicalized Brundu-Logar normal form.

	\subsection{Tropicalizing the 27 lines in the Brundu-Logar normal form}\label{sec:BLNF_lines}
	
	For a cubic $f_P$ in Brundu-Logar normal form, \cite[Table 2]{BL98} gives the Pl\"ucker coordinates of the 27 lines in terms of $(a,b,c,d,e,f)$, where $e$ and $f$ satisfy $g=e+f$, $ac=ef$. 
	
	This offers a new perspective on working the problem of relative realizability: tropicalizing the coordinates we can conclude which of the tropical lines in the tropicalized cubic surface in Brundu-Logar normal form are realizable. Further, for any cubic we can determine the transformation bringing it into normal form, apply its inverse to the $27$ lines and obtain analogous information.
	
	Table \ref{tab:plückercoordinates} shows the tropicalized Pl\"ucker coordinates. We can obtain the vertices of tropical lines from their Pl\"ucker coordinates; see \cite{Ge19}, 
	\cite[Example 4.3.19]{MS15}.
	For shorter notation, we write $v_q$ instead of $\val(q)$.
	\scriptsize	\setlength{\tabcolsep}{3pt}
	\begin{longtable}{ll}	
		$\trop{E_1}  $ &$=[\infty,\infty,0,\infty,\infty,\infty] $\\
		$\trop{G_4} $ & $=[\infty,\infty,\infty,\infty,\infty,0] $\\ 
		$\trop{E_2} $ & $=[\infty,\infty,\infty,0,\infty,\infty] $\\
		$\trop{G_3} $ & $=[0,0,\infty,\infty,0,0] $\\
		$\trop{E_3} $ & $=[\infty,0,0,0,0,\infty] $\\
		$\trop{G_5} $ & $=[2v_c,v_c+v_e,\infty,\infty,v_c+v_e,2v_e] $\\
		$\trop{G_6} $ & $=[2v_c, v_c+v_f,\infty,\infty,v_c+v_f,2v_f] $\\
		$\trop{F_{24}} $ & $=[\infty,\infty,\infty,v_{bc-c^2+ef},v_{-cd+ef+c^2},v_c+v_{e+f-d}] $\\
		$\trop{F_{14}} $ & $=[\infty,v_{bc+c^2+ef},v_{-c^2-cd+ef},\infty,\infty,v_{-b+e+f}+v_c] $\\
		
		$\trop{F_{34}} $ & $=[\infty,v_{-bc+c^2-ef},v_{c^2+cd-ef},v_{-bc+c^2-ef},v_{c^2+cd-ef},v_c+v_{d+b}] $\\
		
		$\begin{aligned}\trop{F_{13}}\end{aligned} $ & $\begin{aligned}=[&v_{bc-c^2+ef},v_{bc-c^2+ef},v_{c-e}+v_{c-f},\infty,v_{bc-ce-cf+2ef},v_{bc-ce-cf+2ef}] \end{aligned}$\\
		
		$\begin{aligned}\trop{F_{15}}\\ \,\end{aligned} $ & $\begin{aligned}=[&v_c+v_{bc-c^2+ef},v_e+v_{bc-c^2+ef},v_{c-e}+v_{cd-cf-ef},\infty,2v_c+v_{e-f-b},\\&v_c+v_e+v_{e-f-b}]\end{aligned} $\\
		$\begin{aligned}\trop{F_{16}} \\ \,\end{aligned}$ & $\begin{aligned}=[&v_c+v_{bc-c^2+ef},v_f+v_{bc-c^2+ef},v_{c-f}+v_{cd-ce-ef},\infty,2v_c+v_{f-e-b},\\&v_c+v_f+v_{f-e-b}] \end{aligned} $\\
		$\begin{aligned}\trop{F_{25}} \\ \,\end{aligned}$ & $\begin{aligned}=[&v_{c^2+cd-ef}+v_c,v_{d-e+f}+2v_c,\infty,v_{c+e}+v_{bc-cf+ef},v_{c^2+cd-ef}+v_e,\\&v_{d-e+f}+v_c+v_e] \end{aligned} $\\
		$\begin{aligned}\trop{F_{26}}\\ \,\end{aligned} $ & $\begin{aligned}=[&v_{c^2+cd-ef}+v_c,v_{d-f+e}+2v_c,\infty,v_{c+f}+v_{bc-ce+ef},v_{c^2+cd-ef}+v_f,\\&v_{d-f+e}+v_c+v_f] \end{aligned}$\\
		
		$\begin{aligned}\trop{F_{23}} \\ \,\end{aligned}$ & $\begin{aligned}=[&v_{c^2+cd-ef},v_{-cd+ce+cf+2ef},\infty,v_{c+e}+v_{c+f},v_{-c^2-cd+ef},\\&v_{cd-ce-cf-2ef}] \end{aligned}$\\
		
		$\begin{aligned}\trop{F_{35}} \\ \,\end{aligned}$ & $\begin{aligned}=[&v_c +v_{bc-cd+2ef},v_{c^2f-c^2d+bce+e^2f},v_{c-e}+v_{cd-cf-ef},\\&v_{c+e}+v_{bc-cf+ef},v_{c^2f-bc^2-cde+e^2f},v_c+v_e+v_{2f-d-b}] \end{aligned}$\\
		
		$\begin{aligned}\trop{F_{36}} \\ \,\end{aligned} $ & $\begin{aligned}=[&v_c +v_{bc-cd+2ef},v_{c^2e-c^2d+bcf+ef^2},v_{c-f}+v_{cd-ce-ef},\\&v_{c+f}+v_{bc-ce+ef},v_{c^2e-bc^2-cdf+ef^2},v_c+v_f+v_{2e-b-d}]  \end{aligned}$\\
		
		$\begin{aligned}\trop{E_5} \\ \,\end{aligned} $ & $\begin{aligned}=[&\infty,v_{c-f}+v_{bc-cf+ef}+v_{cd-cf-ef},v_{f-c}+2v_{cd-cf-ef},v_{c+f}+2v_{bc-cf+ef},\\&v_{c+f}+v_{bc-cf+ef}+v_{cd-cf-ef},v_2+v_{bc-cf+ef}+v_{cd-cf-ef}+v_f]  \end{aligned}$\\
		
		$\begin{aligned}\trop{E_6} \\ \,\end{aligned} $ & $\begin{aligned}=[&\infty,v_{c-e}+v_{bc-ce+ef}+v_{cd-ce-ef},v_{e-c}+2v_{cd-ce-ef},v_{c+e}+2v_{bc-ce+ef},\\&v_{c+e}+v_{bc-ce+ef}+v_{cd-ce-ef},v_2+v_{bc-ce+ef}+v_{cd-ce-ef}+v_e] \end{aligned} $\\
		
		$\begin{aligned}\trop{F_{56}}  \\ \, \\ \,\end{aligned}$ & $\begin{aligned}=[&2v_c+v_{bc-cd+2ef},v_{-bc^2d+bc^2f+bc^2e-c^2ef+fe^2c-cdef+bcef+f^2ec+e^2f^2},\\&v_{cd-ce-ef}+v_{cd-cf-ef},v_{bc-cf+ef}+v_{bc-ce+ef},\\&v_{c^2ef-c^2df-c^2de+bc^2d+fe^2c+f^2ec+cdef-bcef-e^2f^2},v_{bc-cd+2ef}+v_e+v_f] \end{aligned} $	\\
		
		$\begin{aligned}\trop{E_4}  \\ \,\\ \,\end{aligned}$ & $\begin{aligned}=[&v_2+v_{c^2+cd-ef}+v_{bc-c^2+ef}+2v_c,v_{c^2+cd-ef}+v_{bc-c^2+ef}\\&+v_{c^2+ce+cf-ef},2v_{c^2+cd-ef}+v_{c-e}+v_{c-f},v_{c+e}+v_{c+f}+2v_{bc-c^2+ef},\\&v_{c^2+cd-ef}+v_{c^2-ce-cf-ef}+v_{bc-c^2+ef},v_2+v_{c^2+cd-ef}+v_{bc-c^2+ef}+v_e+v_f]  \end{aligned}$\\
		
		$\begin{aligned}\trop{F_{45}} \\ \,\\ \,\end{aligned} $ & $\begin{aligned}=[&\infty ,v_{c+f}+v_{c-f}+v_{cd-cf-ef}+v_{bc-c^2+ef},2v_{c-f}+v_{c^2+cd-ef}+v_{cd-cf-ef},\\&2v_{c+f}+v_{bc-cf+ef}+v_{-bc+c^2-ef},v_{c+f}+v_{c-f}+v_{c^2+cd-ef}+v_{bc-cf+ef},\\&v_{c+f}+v_{c-f}+v_{bc^2+bcf+c^2d-2c^2f-cdf+2ef^2}+v_e]  \end{aligned}$\\	
		
		$\begin{aligned}\trop{F_{46}} \\ \,\\ \,\end{aligned}$ & $\begin{aligned}=[&\infty ,v_{c+e}+v_{c-e}+v_{cd-ce-ef}+v_{bc-c^2+ef},2v_{c-e}+v_{c^2+cd-ef}+v_{cd-ce-ef},\\&2v_{c+e}+v_{bc-ce+ef}+v_{-bc+c^2-ef},v_{c+e}+v_{c-e}+v_{c^2+cd-ef}+v_{bc-ce+ef},\\&v_{c+e}+v_{c-e}+v_{bc^2+bce+c^2d-2c^2e-cde+2e^2f}+v_f]  \end{aligned}$\\
		
		$\begin{aligned}
		\trop{G_1}\\
		\,\\
		\,\\
		\,
		\end{aligned} $ & $\begin{aligned}
		=[&2v_{bc-c^2+ef}+v_{bc-cd+2ef},v_{bc-c^2+ef}+v_{bc-ce-cf+2ef}+v_{bc-cd+2ef},\infty,\\&v_2+v_{bc-c^2+ef}+v_{bc-cf+ef}+v_{bc-ce+ef},\\&v_{bc-c^2+ef}+v_{b^2c^2+bc^2d-bc^2e-bc^2f-c^2de-c^2df+2c^2ef+2cdef-2e^2f^2},\\&v_{bc-ce-cf+2ef}+v_{b^2c^2+bc^2d-bc^2e-bc^2f-c^2de-c^2df+2c^2ef+2cdef-2e^2f^2}] \end{aligned}$\\
		
		$\begin{aligned}\trop{F_{12}}  \\ \, \\ \, \\ \, \\ \,\end{aligned}$ & $\begin{aligned}=[&2v_{bc-cd+2ef}+ v_{c^2+cd-ef}+v_{-bc+c^2-ef},v_{bc-cd+2ef}\\&+v_{-2bcef-2e^2f^2+(-b-d+2e)c^2f+(b+d)(d-e)c^2}+v_{bc-c^2+ef},\infty,\infty,\\&v_{c^2+cd-ef}+v_{bc-cd+2ef}+v_{2cdef-2e^2f^2-(b+d-2e)c^2f-(-b+e)(b+d)c^2},\\&v_{bc^2d-bc^2e-bc^2f-2bcef+c^2d^2-c^2de-c^2df+2c^2ef-2e^2f^2}\\&+v_{b^2c^2+bc^2d-bc^2e-bc^2f-c^2de-c^2df+2c^2ef+2cdef-2e^2f^2}]  \end{aligned}$\\
		
		$\begin{aligned}\trop{G_2}  \\ \,\\ \,\\ \,\\ \,\end{aligned}$ & $\begin{aligned}=[&2v_{c^2+cd-ef}+v_{bc-cd+2ef},v_{c^2+cd-ef}\\&+v_{bc^2d-bc^2e-bc^2f-2bcef+c^2d^2-c^2de-c^2df+2c^2ef-2e^2f^2},\\&v_2+v_{c^2+cd-ef}+v_{cd-ce-ef}+v_{cd-cf-ef},\infty,v_{c^2+cd-ef}\\&+v_{cd-ce-cf-2ef}+v_{bc-cd+2ef},v_{cd-ce-cf-2ef}\\&+v_{bc^2d-bc^2e-bc^2f-2bcef+c^2d^2-c^2de-c^2df+2c^2ef-2e^2f^2}]  \end{aligned}$\\
		\caption{\normalsize The tropical Pl\"ucker coordinates of the 27 tropicalized lines on the tropicalized Brundu-Logar normal form 
		}\label{tab:plückercoordinates}
	\end{longtable}
	\normalsize
	
	\begin{rem}\label{rem:lines_in_infty}
	Note that the tropicalizations of the first five lines $E_1,$ $G_4,$ $E_2,$ $G_3$ and $E_3$ are always distinct. These are depicted in Figures \ref{fig:5infties}, \ref{fig:2infties}.	
	
	As $23$ of the tropicalized lines on a tropicalized surface in Brundu-Logar normal form contain at least one  tropical Pl\"ucker coordinate of value $\infty$, here is a short overview how to visualize these lines: with five Pl\"ucker coordinates of value $\infty$ the tropical line is \glqq an edge\grqq of the tetrahedron $T$ modelling the tropical projective space \cite[Chapter 6]{MS15}, as shown in Figure \ref{fig:5infties}. The case of four infinite Pl\"ucker coordinates does not occur naturally in Table \ref{tab:plückercoordinates}; it would lead to a line segment in one of the facets of $T$. Three Pl\"ucker coordinates of value $\infty$ lead to a tropical half line visible in one of the facets of $T$ as in Figure \ref{fig:3infties}. Two Pl\"ucker coordinates of value $\infty$ lead to a classical line in the interior of $T$, corresponding to the middle line segment of a tropical line of infinite length in the tropical surface; see for example Figure \ref{fig:2infties}. One Pl\"ucker coordinate of value $\infty$ leads to a tropical half line visible in the finite parts of the surface. This is depicted for Example \ref{ex:trivsub} in Figure \ref{fig:halflines}.

	\end{rem}
\begin{figure}
    \centering
    \begin{subfigure}{.3\textwidth}
      \includegraphics[width=.9\linewidth]{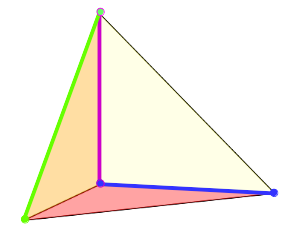}
      \caption{ Lines in $\infty$ \\blue: $\trop{E_1}$,\\ green: $\trop{E_2}$, \\magenta: $\trop{G_4}$}\label{fig:5infties}
    \end{subfigure}
    \begin{subfigure}{.3\textwidth}
      \includegraphics[width=.9\linewidth]{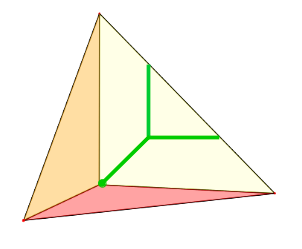}
      \caption{Half line in $\infty$,\\ green: $\trop{F_{14}}$ \\ \ \\ \ }\label{fig:3infties}
    \end{subfigure}
    \begin{subfigure}{.3\textwidth}
      \includegraphics[width=.9\linewidth]{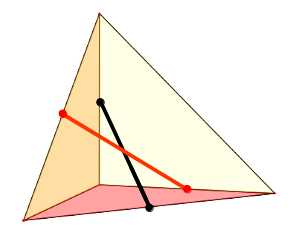}
      \caption{Vertices in $\infty$, \\red: $\trop{G_3}$,\\ black: $\trop{E_3}$ \\ \ }\label{fig:2infties}
    \end{subfigure}    
    \caption{The tropical projective space. The facets are as follows: orange  $x=\infty$, yellow $y= \infty$, red $z=\infty.$}
    \label{fig:infinitelines}
\end{figure}{}	

Since at least $23$ of the tropicalized lines have Pl\"ucker coordinates of value $\infty$, at most $4$ classical lines on $S_p$ can have tropicalizations fully visible on $\trop{S_P}.$
In comparison, of the $27$ lines on a cubic in octanomial form, only one classical line has finite tropical Pl\"ucker coordinates \cite{PSS19}.

For a given surface $S_P$ in Brundu-Logar normal form we can compute all completely visible tropical lines on $\trop{S_P}$ using the \texttt{polymake} extension \texttt{a-tint} by Simon Hampe \cite{Ha14,polymake:2000}, and compare which come from a classical line on $S_P$. 

This is carried out for a sample cubic in Example \ref{ex:trivsub}. 
We obtain higher lifting multiplicities in this setting than for the octanomial form \cite{PSS19}. Example \ref{ex:1} shows that for the Brundu-Logar normal form even isolated lines can have higher lifting multiplicities.

\section{Examples}\label{sec:ex}
In this section we show some examples on the theory presented in Sections \ref{sec:relreal} and \ref{sec:BLNFandtrop}. We start with two examples on the relative realizability of families of motif 3J and 3I-1 indicating positive answers to Questions~\ref{que:3I-1} and \ref{que:3J}. We conclude with examples investigating the behavior of lines on the tropicalized Brundu-Logar normal form.

\subsection{Relative Realizability}
 
	\begin{exa}[Lines of motif 3J]\label{ex:3J} This example indicates that Question~\ref{que:3J} might be answerable with yes. In fact it shows a cubic surface on which a member of a family of motif 3J is realizable.
	
	We investigate the cubic surface detected by Hampe and Joswig in \cite{HJ17} containing two families of motif 3I and one of motif 3J. We choose the weight vector $\omega$ as below, which induces a surface containing no lines of non-general motifs and can therefore be considered as generic:
$$\omega:=(143,0,64,122,0,2,0,15,55,107,36,23,39,16,14,48,12,12,49,95).$$
		Using the Magma script \cite{Se19} by Emre Sertoz written for \cite{PSS19} we can compute the $27$ lines on a simple lift of the tropical cubic over an $p$-adic field. 
		A simple lift is a cubic where the coefficients $c_i$ are given as $t^{\omega_i}$, where $t$ is either a prime if we are over the $p$-adics, or $t\in\mathbb{C}\{\{t\}\}$ over the Puiseux series. 
		The following table shows the tropicalized Pl\"ucker coordinates of the computed $27$ lines for $p=2$ matching those for $p=5$, indicating that $p=2$ does not have a special role in the lifting behavior of this case. The table also shows the Pl\"ucker coordinates of the tropical lines contained in the tropical surface sorted into isolated lines and families as computed by the \texttt{polymake} extension \texttt{a-tint} \cite{Ha14,HJ17}. 
	\footnotesize{	\setlength{\tabcolsep}{3pt}\begin{longtable}{c|cccccc|c|cccccc|c}
			\text{} & \multicolumn{6}{c|}{\small{tropicalized lines}} & \text{} & \multicolumn{6}{c|}{\small{tropical lines}} & \small{motif}\\
			\hline 
			\endfirsthead   
			\text{} & \multicolumn{6}{c|}{tropicalized lines} & \text{} & \multicolumn{6}{c|}{tropical lines} & motif\\
			\hline 
			\endhead
			$1$& $0$ &$-40$& $-86$& $37$& $3$ & $-49$  & \multirow{20}{*}{\rotatebox{270}{\scriptsize{isolated}}} &	$0$ &$-40$& $-86$& $37$& $3$ & $-49$ &  3B            \\
			$2$&  $0$ & $-3$& $-49$& $0$ &$-37$& $-49$ &  &$0$ & $-3$& $-49$& $0$ &$-37$& $-49$    & 3A  \\ 
			$3$&  $0$ &$22$ &$0$ &$-12$& $-12$& $-12$ &  & $0$ &$22$ &$0$ &$-12$& $-12$& $-12$ & 3G    \\
			$4$&  $0$& $22$& $0$& $0$& $0$& $0$&  &$0$& $22$& $0$& $0$& $0$& $0$   &   3D      \\
			$5$& $0$ &$2$& $0$& $-12$& $-12$& $-12$ & & $0$ &$2$& $0$& $-12$& $-12$& $-12$  &      3G   \\
			$6$&  $0$& $2$& $0$& $0$& $0$& $0$&  & $0$& $2$& $0$& $0$& $0$& $0$ &     3D           \\
			$7$&  $0$& $-49$& $-86$& $37$& $3$& $-49$ &  & $0$& $-49$& $-86$& $37$& $3$& $-49$ &     3H    \\
			$8$&  $0$& $2$ &$-30$& $-2$& $-48$& $-46$ &   &$0$& $2$ &$-30$& $-2$& $-48$& $-46$ &   3F  \\
			$9$&  $0$& $2$& $-30$& $-11$& $-48$& $-46$& &$0$& $2$& $-30$& $-11$& $-48$& $-46$ &   3F  \\
			$10$&  $0$ &$-12$ &$-12$& $46$& $0$ &$-12$ &  & $0$ &$-12$ &$-12$& $46$& $0$ &$-12$ &    3G     \\
			$11$&  $0$ &$-12$ &$-12$ &$37$ &$0$ &$-12$ &   & $0 $&$-12$ &$-12$ &$37$ &$0$ &$-12$&3G     \\
			$12$&  $0$ &$-12$ &$-49$ &$9$ &$-37$& $-49$ & & $0$ &$-12$ &$-49$ &$9$ &$-37$& $-49$ &   3A   \\
			$13$&  $0$ &$0$& $-36$& $-2$& $-48$ &$-48$ &   &$0$ &$0$& $-36$& $-2$& $-48$ &$-48$&  3D  \\
			$14$& $0$& $0$& $-36$& $-11$& $-48$& $-48$&    & $0$& $0$& $-36$& $-11$& $-48$& $-48$ &   3D  \\
			$15$& $0$& $22$& $-12$& $-2$& $-48$& $-26$ &  & $0$& $22$& $-12$& $-2$& $-48$& $-26$&    3D   \\
			$16$&  $0$ &$22$ &$-12$ &$-11$ &$-48$ &$-26$& &$0$ &$22$ &$-12$ &$-11$ &$-48$ &$-26$ &  3D   \\
			$17$&  $0$ &$22$& $0$& $46$ &$0$ &$22$&   & $0$ &$22$& $0$& $46$ &$0$ &$22$&  3D    \\
			$18$&  $0$ &$22$& $0$ &$37$& $0$ &$22$ &    & $0$ &$22$& $0$ &$37$& $0$ &$22$&   3D      \\
			$19$&  $0$ &$0$ &$0$ &$46$& $0$& $0$ &   &$0$ &$0$ &$0$ &$46$& $0$& $0$ &     3D      \\
			$20$& $0$& $0$& $0$& $37$& $0$ &$0$&   &$0$& $0$& $0$& $37$& $0$ &$0$ &     3D      \\
			$21$&  $0$ &$2$& $0$& $46$ &$0$ &$2$ &   & $0$ &$2$& $0$& $46$ &$0$ &$2$ &      3D  \\
			$22$&  $0$ &$2$ &$0$ &$37$& $0$& $2$ &   & $0$ &$2$ &$0$ &$37$& $0$& $2$ & 3D     \\
			$23$&  $0$ &$20$& $2$& $-16$ &$-16$ &$-14$&   & $0$ &$20$& $2$& $-16$ &$-16$ &$-14$& 3H	\\
			$24$& $0$ &$-61$& $-107$& $52$& $9$& $-55$ &  &$0$ &$-61$& $-107$& $52$& $9$& $-55$&3H		\\ 
			$25$&  $0$& $-55$& $-55$& $49$ &$3$ &$-52$ &    & $0$& $-55$& $-55$& $49$ &$3$ &$-52$& 3H\\
	        $26$&	$0$& $-55$& $-92$& $49$& $3$ &$-52$&  &$0$& $-55$& $-92$& $49$& $3$ &$-52$& 3H	\\
				\hline
			$27$&$0$ &$0$ &$0$& $0$& $0$& $2$&  \multirow{3}{*}{\rotatebox{270}{\scriptsize{fam.}}}  &$0$ &$0$&$0$ &$0$ &$0$ &$t$ & 3J, \scriptsize{$t\geq 0$} 		\\
			& & & & & & &    &$12+t$& $0$& $0$& $12$& $12$& $0$& 3I-2, \scriptsize{$t\geq0$} 	\\
			& & & & & & &             &$12+t$ &$12$ &$12$ &$0$&$0$ &$0$ & 3I-2, \scriptsize{$t\geq0$} 	\\
			\caption{The tropicalized lines compared with the tropical lines on the tropicalization of a simple lift over $\mathbb{Q}_2$ to the weight vector $\omega$}
		\end{longtable}}
\normalsize	
\vspace{-1pt}
		Investigating the dual subdivision of the surface, we see that the two families of motif 3I have type 3I-2. As proven in Theorems \ref{theorem:relreal} and \ref{theorem:deg3I} no member of these families lifts. In particular, all the 27 lines on this simple lift tropicalize to non-degenerate lines. It is notable that all the 26 isolated lines are realizable on our chosen lift of the surface. Furthermore, for this example the statement from \cite[Conjecture 4.1]{PSS19} holds: the tropicalizations of the $27$ lines are distinct.
	
		Notice that in this example the member for $t=2$ of the family of motif 3J lifts, supporting the idea that we have a different lifting behavior for families of motif 3J than for motif 3I; see Question \ref{que:3J}.

	\end{exa}
	\begin{exa}[Lines of type 3I-1 over $\text{char}(\Bbbk)=2$]\label{ex:3I-1} We consider an example of a smooth tropical cubic surface containing a family of type 3I-1 and its lifting behavior over $\mathbb{Q}_2$ and $\mathbb{Q}_3$. The results suggest Question \ref{que:3I-1} might be answered positively and Proposition \ref{cor:relreal} might be extendable to characteristic $p$. We also compare the lifting of the lines with the corresponding Brundu-Logar normal form to the chosen surface.
	
		Using the Magma script 
		from \cite{Ku19}, we can calculate the Brundu-Logar normal form to any generic cubic over the $p$-adics and investigate the behavior of the tropical lines on both sides. We consider the tropical cubic surface given by the weight vector $\nu$ below. This cubic is of interest as it contains a family of type 3I-1. We investigate a simple lift $S_{\nu}$ of this surface over $\mathbb{Q}_2$, giving an example corroborating a positive answer to Question \ref{que:3I-1}.
		$$\nu=( 0 , 12 , 30 , 0 , 30 , 12 ,3, 145, 54, 10, 51, 0 , 18, 123 , 30, 0 , 265 , 150, 80, 21).$$
		The tropicalized Brundu-Logar normal form, $\trop{S^{nf}_{\nu}}$, to this cubic contains $4$ families, numbered $F_i$, $i=0,...,3$. We have the weight vector $$\nu^{nf}=(-10, -6, -6, -10, -10, -10, -8, -10, -10, -9).$$  
		
		\begin{figure}[h]
		    \centering
		    \begin{minipage}{.25\textwidth}
		    \centering
		    \includegraphics[height=0.8\linewidth]{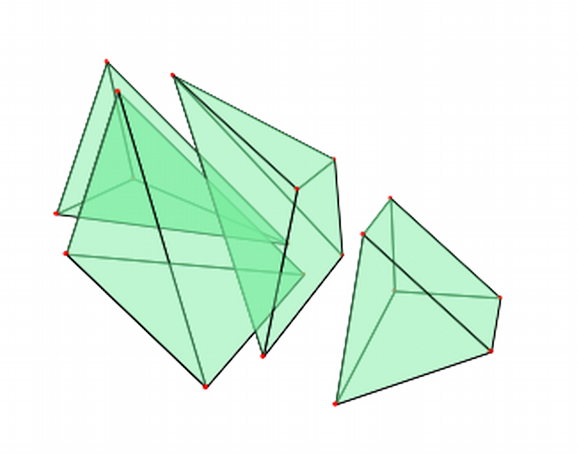}
		      \caption{\\ Subdiv($\trop{S_{\nu}^{nf}}$)}
		    \label{fig:dualsubdiv}
		    \end{minipage}
		    \begin{minipage}{.6\textwidth}
		    \centering
		 \footnotesize	\setlength{\tabcolsep}{3pt}\begin{longtable}{c|cccccc|c}
			\text{} & \multicolumn{6}{c}{tropical Pl\"ucker coordinates} &\text{} \\
			\hline 
			\endfirsthead   
			\text{} & \multicolumn{6}{c|}{Pl\"ucker coordinates} &\text{}\\
			\hline 
			\endhead
			$F_0$&  $0$&  $0$&   $\lambda$ &  $1$ & $0$ & $0$ & $\lambda\geq 0$
		                  \\
		                  \hline
			$F_1$&  $0$&  $1$&   $0$ &  $1$  & $0$ &  $1+\lambda$    & $\lambda\geq 0$          \\ 
		\hline
			$F_2$&  $0$&  $1+\lambda$&   $0$ &$1$  & $0$ &  $1$     & $\lambda\geq 0$     \\
		\hline
			\multirow{2}{*}{	$F_3$}	 & $1-\lambda$&  $0$&  $0$ &  $2-\lambda$ &  $1-\lambda$ &$0$ & $\lambda\in[0,1]$\\
				&  $0$& $\lambda$&  $0$ &  $1$ & $0$& $\lambda$  & $\lambda\in[0,1]$  \\
			\caption{Tropical lines on $\trop{S^{nf}_{\nu}}$}\label{tab:BLNFlines3I-1}
		\end{longtable}	 \end{minipage}
		\end{figure}\normalsize

	Computing the tropicalizations of the lines in the Brundu-Logar normal form $S^{nf}_{\nu}$ (right columns in Table \ref{tab:exHJ}) and their inverse images under the transformation leading to the normal form (left columns), we can investigate the lifting behaviour of the tropical lines on both surfaces.
	
	Investigating the Pl\"ucker coordinates of the tropicalized lines on $\trop{S^{nf}_{\nu}}$, as demonstrated in Table \ref{tab:exHJ}, we obtain two tropicalized lines of higher lifting multiplicity on $\trop{S^{nf}_{\nu}}$: a tropical half line of lifting multiplicity $3$ and a middle line segment of infinite length of lifting multiplicity $4$. The other $20$ lines on $S^{nf}_{\nu}$ have distinct tropicalizations. Comparing Table \ref{tab:BLNFlines3I-1} with the middle right column of Table \ref{tab:exHJ}, we obtain that three inner members of the family $F_0$ lift and that one inner member of family $F_1$ lifts, while no members of $F_2$ and $F_3$ lift. Additionally, $F_0$ approximates the tropicalized line of lifting multiplicity $3$ for $\lambda\to\infty$. 
			\footnotesize{	\setlength{\tabcolsep}{2pt}\begin{longtable}{cccccc|c||cccccc|c}
			\multicolumn{6}{c|}{lines on $\trop{S_{\nu}}$} & \text{lifting?} & \multicolumn{7}{c}{lines on $\trop{S^ {nf}_{\nu}}$} \\
			\hline 
			\endfirsthead   
			 \multicolumn{6}{c|}{lines on $\trop{S_{\nu}}$} & \text{lifting?} & \multicolumn{6}{c}{lines on $\trop{S^ {nf}_{\nu}}$} &\text{}\\
			\hline 
			\endhead
$0$& $0$& $0$& $t$& $0$& $0$& (3I-1), yes $t=1$	&$\infty$&$\infty$& $\infty$& $\infty$& $\infty$& $0$ & (e)	
\\
$0$ &$-40$&$-20$& $-30$ &$-10$& $-40$& yes   & $\infty$& $\infty$& $0$&$\infty$& $\infty$& $\infty$ & (e)	
\\
$0$ &$-123$& $-21$& $-111$&$-9$& $-123$&  yes  &$\infty$&$\infty$& $\infty$& $0$&$\infty$& $\infty$ & (e)	
\\
$0$ &$-41$ &$-21$& $-31$& $-11$& $-41$&yes     &$\infty$& $-1$& $-4$& $\infty$&$\infty$& $-3$ &	(d)	
\\
$0$ &$-123$& $-22$& $-111$& $-9$& $-133$& yes  &$0$& $0$& $82$&$\infty$& $0$& $0$ &	(b)
\\
$0$ &$-45$& $-22$& $-34$& $-10$& $-56$&yes     &$0$& $1$&$\infty$& $1$& $0$& $1$ & (b)
\\
$0$ &$-45$& $-22$& $-34$& $-4$& $-56$&   yes  &$0$& $1$& $0$& $1$&$0$& $7$ & (a), $F_1$:  $\lambda =6$
\\
$0$ &$-48$& $-23$& $-35$& $-10$& $-48$&yes     &$\infty$& $0$& $0$& $0$& $0$& $\infty$ & (c)	
\\
$0$& $-48$& $-23$& $-36$& $-11$& $-58$&  yes  &$0$& $0$& $\infty$& $\infty$& $0$& $0$ &(c)
\\
$0$& $-123$& $-23$& $-111$& $-9$& $-134$&  yes &$0$& $0$& $81$& $1$& $0$& $0$ &	(a), $F_0$: $\lambda=81$
\\
$0$& $-48$& $-23$& $-36$& $-11$& $-48$& yes &$\infty$& $0$& $1$& $0$& $1$& $0$ &(b)
\\
$0$& $-48$& $-23$& $-35$& $-5$& $-58$& yes  &$0$& $0$& $1$& $\infty$& $0$& $0$ &  (b)
\\
$0$& $-127$& $-27$& $-109$& $-11$& $-138$& yes  &$0$& $0$& $\infty$& $1$& $0$& $0$ & (b), $F_0$:  $\lambda\to\infty$ 
\\
$0$& $-128$& $-27$& $-110$& $-10$& $-138$& yes &$0$& $0$& $\infty$&$\infty$& $0$& $0$ &(c)
\\
$0$& $-129$& $-27$& $-111$& $-9$& $-129$& yes &$\infty$& $\infty$& $\infty$& $0$& $0$& $0$ &	(d)
\\
$0$& $-75$& $-30$& $-63$& $-18$& $-86$& yes &$0$& $0$& $\infty$& $1$& $0$& $0$ & (b), $F_0$: $\lambda\to\infty$ 
\\
$0$& $-75$& $-30$& $-63$& $-18$& $-75$&yes  &$\infty$& $20$& $41$& $0$& $21$& $20$ &(b)
\\
$0$& $-75$& $-30$& $-63$& $-18$& $-85$& yes &$0$& $0$& $\infty$& $\infty$& $0$& $0$ &(c)
\\
$0$& $-114$& $-30$& $-102$& $-18$& $-114$&yes &$\infty$& $62$& $122$& $3$& $63$& $62$&(b)
\\
$0$& $-114$& $-30$& $-102$& $-18$& $-124$&yes &$0$& $0$& $\infty$&$\infty$& $0$& $0$ &(c)
\\
$0$& $-114$& $-30$& $-102$& $-18$& $-125$&yes &$0$& $0$& $\infty$&$1$& $0$&$0$&(b), $F_0$: $\lambda\to\infty$
\\
$0$& $-120$& $-36$& $-102$& $-18$& $-131$&yes  &$0$& $0$& $61$& $1$& $0$& $0$ & (a), $F_0$: $\lambda=61$	
\\
$0$& $-81$& $-36$& $-63$& $-18$& $-91$& yes &$0$& $0$& $21$& $\infty$& $0$& $0$&(b)
\\
$0$& $-81$& $-36$& $-63$& $-18$& $-81$& yes &$\infty$& $20$& $20$& $0$& $0$& $0$&(b)
\\
$0$& $-81$& $-36$& $-63$& $-18$& $-92$& yes &$0$& $0$& $22$& $1$& $0$& $0$ &(a), $F_0$: $\lambda=22$
\\
$0$ &$-120$& $-36$& $-102$& $-18$& $-130$& yes &$0$& $0$& $60$&$\infty$& $0$& $0$ &	(b)
\\
$0$& $-120$& $-36$& $-102$& $-18$& $-120$& yes &$\infty$& $59$& $59$& $0$& $0$& $0$ &(b)
\\
$0$&$-48$&$-23$&$-36+t$&$-11$&$-59$& (3I-2), no& & & & & & & \\

\caption{On the last column, cf. Remark \ref{rem:lines_in_infty}:\\ (a) \glqq visible \grqq tropical line, \\(b) tropical half line (see Figure \ref{fig:halflines}), \\(c) middle line segment with vertices in infinity (see Figure \ref{fig:2infties}), \\(d) tropical half line in infinite boundary (see Figure \ref{fig:3infties}), \\(e) tropical line in the infinite boundary (see Figure \ref{fig:5infties}).}\label{tab:exHJ}
		\end{longtable}}
\normalsize	

Using \texttt{polymake}, we obtain that the tropical cubic surface $\trop{S_{\nu}}$ contains $26$ isolated lines. They are all non-degenerate and lift onto the chosen simple lift of the surface. Additionally, the surface contains two families, both of motif 3I. However, an investigation of the dual subdivision shows that one family is of type 3I-1, while the other is of type 3I-2. As proven in Theorems \ref{theorem:relreal} and \ref{theorem:deg3I} the family of type 3I-2 does not lift. Nevertheless, we are in the case that the characteristic of the residue field is equal to $2$, and we observe that for $t=1$ a member of the family of type 3I-1 does lift in this setting.

Comparing these results with a simple lift over $\mathbb{Q}_3$, we obtain that in this case in accordance with Theorem \ref{theorem:relreal} the non-degenerate lines of type 3I-1 do not lift. However, the degenerate line of type 3I-1 lifts, suggesting Proposition \ref{cor:relreal} might hold in different characteristics; see Remark \ref{rem:3I-1}. 
The isolated lines have the same lifting behavior for $p=2$ and $p=3$.

	\end{exa}

\subsection{Lines on the tropicalized Brundu-Logar normal form} 
In the following let $K=\mathbb{C}\{\{t\} \}$ be the field of Puiseux series over $\mathbb{C}$.
\begin{exa}\label{ex:trivsub} This example is tropically very degenerate. Most of the tropicalized lines vanish completely or partly into the infinite boundary, but we observe that some of them can be approximated by families. We also obtain higher lifting multiplicities.

Choose $P=(a,b,c,d,g)\in\mathbb{P}^4$ as $a= \frac{1}{2}(1+t+t^{2})$, $b= t^5$, $c= 2+2t^2$, $d= t^{10}$, $e=1$, $f= 1+t+2t^2+t^3+t^4$, where $g=e+f$.\\ This choice satisfies $ac=ef$ and $\sigma(P)\neq 0$, so $P$ encodes the Brundu-Logar normal form of some smooth cubic surface. These parameters lead to the trivial subdivision of $\text{Newt}(\trop{S_P})$; the dual surface is shown in Figure \ref{fig:surface_trivsub}. 
In this case we identify the tropicalizations of all $27$ lines.
The tropical surface $\trop{S_P}$ contains no isolated lines and 7 families, whose tropical Pl\"ucker coordinates are as follows with $\mu,\lambda\geq 0$:

\footnotesize	\setlength{\tabcolsep}{3pt}\begin{longtable}{c|cccccc}
			\text{} & \multicolumn{6}{c}{Pl\"ucker coordinates}  \\
			\hline 
			\endfirsthead   
			\text{} & \multicolumn{6}{c}{Pl\"ucker coordinates}\\
			\hline 
			\endhead
			$F_0$&  $\mu$&  $0$&  $0$ & $0$ & $0$ & $\lambda$ \\
		                  \hline
			$F_1$&  $0$&   $\mu$&   $0$ &  $0$ &  $\lambda$ & $0$  \\ 
		\hline
			$F_2$&  $0$&  $0$&   $\mu$ &$\lambda$  & $0$ & $0$          \\
		\hline
			\multirow{3}{*}{	$F_3$}	 & $\lambda+\mu$&  $0$&  $0$ &  $\lambda$ &  $\lambda$ &$0$ \\
				& $\lambda$&  $0$&  $0$ &  $\lambda+\mu$ &  $\lambda$ &$0$   \\
				& $\lambda$&  $0$&  $0$ &  $\lambda$ &  $\lambda+\mu$ &$0$      \\
				\hline
				\multirow{3}{*}{	$F_4$}&  $\lambda+\mu$ &   $\lambda$&   $\lambda$ &  $0$ & $0$ &$0$    \\
				&  $\lambda$ &     $\lambda+\mu$&     $\lambda$ &  $0$ & $0$ & $0$  \\
				&  $\lambda$ &     $\lambda$&     $\lambda+\mu$ &  $0$ & $0$ & $0$  \\
					\hline
		\multirow{3}{*}{	$F_5$}&  $0$&  $0$ &    $\lambda+\mu$&$0$ &   $\lambda$ & $\lambda$      \\
		& $0$&  $0$ &    $\lambda$&$0$ &   $\lambda+\mu$ & $\lambda$   \\
		&$0$&  $0$ &    $\lambda$&$0$ &   $\lambda$ & $\lambda+\mu$  \\
			\hline
		\multirow{3}{*}{	$F_6$}&   $0$&$\lambda+\mu$&   $0$&  $\lambda$&  $0$ &$\lambda$     \\
&   $0$&$\lambda$&   $0$&  $\lambda+\mu$&  $0$ &$\lambda$ \\
&   $0$&$\lambda$&   $0$&  $\lambda$&  $0$ &$\lambda+\mu$ \\
			\caption{The tropical Pl\"ucker coordinates of the families on $\trop{S_P}$}
		\end{longtable}
		
		\normalsize
	
The tropical Pl\"ucker coordinates of the $27$ lines are shown in Table~\ref{tab:expluck}. We see that the $27$ lines tropicalize to $14$ distinct lines, of which only one has finite Pl\"ucker coordinates. The maximal lifting multiplicity observed in this example is $4$; see Table~\ref{tab:expluck}.

By taking the limit of $\lambda$ and/or $\mu\to\infty$ while fixing the other parameter accordingly,
we can approximate some tropicalized lines as an infinite border point of some of the families; see last column of Table~\ref{tab:expluck}.
The lines $E_1,G_4,E_2$, which are completely in infinity, as shown in Figure \ref{fig:infinitelines},
are no limit of any of the families. The tropical lines $\trop{F_{15}}=\trop{F_{16}}$ and $\trop{F_{25}}=\trop{F_{26}}$ have one finite vertex; the other one is in infinity, cf. Remark \ref{rem:lines_in_infty}. So only a tropical half line is visible; see Figure~\ref{fig:halflineF15}-\ref{fig:halflineF25}. This also does not appear as a limit of the families in this case.\\
\
\\

	\footnotesize	\setlength{\tabcolsep}{3pt}\begin{longtable}{c|cccccc|c}
			\text{} & \multicolumn{6}{c|}{Pl\"ucker coordinates} &  \small{} \\
			\hline 
			\endfirsthead   
			\text{} & \multicolumn{6}{c|}{Pl\"ucker coordinates} &  \small{}\\
			\hline 
			\endhead
			$\trop{E_1}$&  $\infty$&   $\infty$&   $0$ &  $\infty$ &  $\infty$ & $\infty$ & \multirow{3}{*}{not approximable}
		                  \\
		  \cline{1-7}
			$\trop{G_4}$&  $\infty$&   $\infty$&   $\infty$ &  $\infty$ &  $0$ & $\infty$ &                \\ 
		 \cline{1-7}
			$\trop{E_2}$&  $\infty$&  $\infty$&   $\infty$ & $0$ & $\infty$ &  $\infty$ &          \\
			\hline
				$\trop{E_3}$& $\infty$&  $0$&  $0$ &  $0$ &  $0$ & $\infty$ & $F_0$ with $\lambda,\mu \to \infty$      \\
				\hline
					$\trop{F_{24}}$&  $\infty$ &     $\infty$&    $\infty$ &  $0$ & $0$ & $0$ &  $F_4$ with $\lambda \to \infty$    \\
					\hline
			$\trop{F_{14}}$&   $\infty$&  $0$&  $0$ &    $\infty$ &   $\infty$ & $0$ &     $F_3$ with $\lambda\to\infty$      \\
			\hline
			$\trop{F_{34}}$&    $\infty$&   $0$&  $0$ &  $0$ & $0$ & $5$ &   $F_0$ with $\mu\to\infty$, $\lambda=5$   \\
			\hline
					$\trop{F_{13}}$&  \multirow{2}{*}{$0$}&    \multirow{2}{*}{$0$}&   \multirow{2}{*}{$0$} &     \multirow{2}{*}{$\infty$} & \multirow{2}{*}{$0$} &  \multirow{2}{*}{$0$} &    $F_2$ with $\lambda\to\infty$, $\mu=0$   \\
			$\trop{G_2}$&   &   &  &     &  &  &     $F_3$ (b) and $F_6$ (b) with $\lambda=0$, $\mu\to\infty$          
			\\
			\hline
			$\trop{F_{15}}$&   \multirow{2}{*}{$0$}&    \multirow{2}{*}{$0$}&   \multirow{2}{*}{$0$} &     \multirow{2}{*}{$\infty$} & \multirow{2}{*}{$1$} &  \multirow{2}{*}{$1$} &    \multirow{4}{*}{not approximable} \\
			$\trop{F_{16}}$&     &  &  &   &  &  &   \\
			\cline{1-7}
			$\trop{F_{25}}$&  \multirow{2}{*}{$0$}&    \multirow{2}{*}{$1$}&   \multirow{2}{*}{$\infty$} &     \multirow{2}{*}{$0$} & \multirow{2}{*}{$0$} &  \multirow{2}{*}{$1$} &       \\
			$\trop{F_{26}}$&  &  &   &  & &  &               \\
			\hline
			$\trop{F_{23}}$&  \multirow{2}{*}{$0$}&    \multirow{2}{*}{$0$}&   \multirow{2}{*}{$\infty$} &     \multirow{2}{*}{$0$} & \multirow{2}{*}{$0$} &  \multirow{2}{*}{$0$} &     $F_2$ with $\mu\to\infty$, $\lambda=0$              \\
			$\trop{G_1}$&   &   &   &  &  & &   $F_4$ (c) and $F_5$ (a) with $\lambda=0$, $\mu\to\infty$		\\
				\hline
			$\trop{G_3}$&   \multirow{4}{*}{$0$}&  \multirow{4}{*}{$0$}&  \multirow{4}{*}{$\infty$} & \multirow{4}{*}{$\infty$}& \multirow{4}{*}{$0$}&  \multirow{4}{*}{$0$}&     \multirow{4}{*}{$F_2$ with $\lambda,\mu\to\infty$}    \\
				$\trop{G_5}$&    & &   &  &  & &         \\
			$\trop{G_6}$&    & &   &  &  & &                  \\
			$\trop{F_{12}}$&    & &   &  &  & &   
			\\ \hline
			$\trop{E_5}$&  \multirow{4}{*}{$\infty$}&    \multirow{4}{*}{$0$}&   \multirow{4}{*}{$0$} &     \multirow{4}{*}{$0$} & \multirow{4}{*}{$0$} &  \multirow{4}{*}{$0$} &  \\
			$\trop{E_6}$& &   &   &   &  &   & 	$F_0$ with $\mu\to\infty$, $\lambda=0$   \\
			
			$\trop{F_{45}}$&  &   &   &   &  &   &   $F_3$ (a) and $F_4$ (a) with $\lambda=0$, $\mu\to\infty$	     \\
			$\trop{F_{46}}$&  &   &   &   &  &   &              \\
			\hline
			$\trop{F_{35}}$&  \multirow{4}{*}{$0$}&    \multirow{4}{*}{$0$}&   \multirow{4}{*}{$0$} &     \multirow{4}{*}{$0$} & \multirow{4}{*}{$0$} &  \multirow{4}{*}{$0$} &  \multirow{4}{*}{all families with $\lambda,\mu=0$} \\
			$\trop{F_{36}}$&  &   &   &   &  &   &       \\
			$\trop{F_{56}}$&  &   &  &   &  &   &      \\
			$\trop{E_4}$&  &   &   &   &  &   &    \\
			\caption{The tropical Pl\"ucker coordinates of the tropicalizations of the lines on $S_P$ and their connection to the tropical lines on $\trop{S_P}$ }\label{tab:expluck}
		\end{longtable}
\normalsize
Note, that the \texttt{a-tint} output for the computation of families has to be checked for redundancies.
\end{exa}

\begin{figure}[ht]
\centering
\begin{subfigure}{.32\textwidth}
  \centering
 \includegraphics[height=0.8\textwidth]{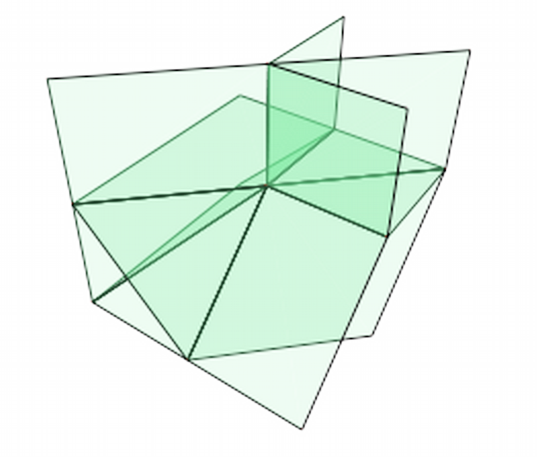}
    \caption{Surface dual to\\ the trivial subdivision}
    \label{fig:surface_trivsub}
\end{subfigure}
   \begin{subfigure}{.32\textwidth}
  \centering
 \includegraphics[height=0.8\textwidth]{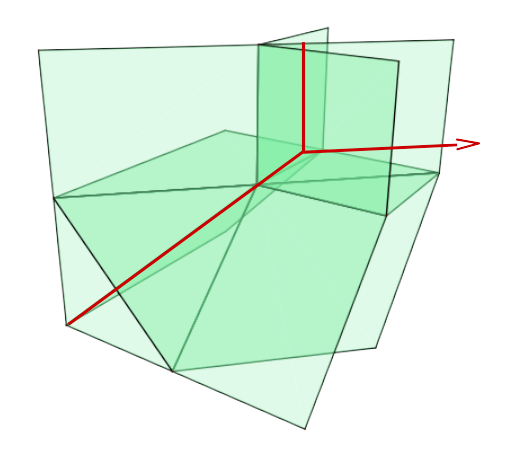} 
    \caption{The tropical half line\\ $\trop{F_{15}}$}
    \label{fig:halflineF15}
\end{subfigure}
   \begin{subfigure}{.32\textwidth}
  \centering
 \includegraphics[height=0.8\textwidth]{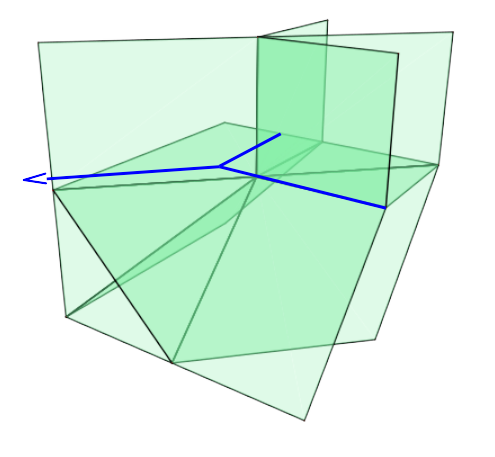} 
    \caption{The tropical half line\\ $\trop{F_{25}}$}
    \label{fig:halflineF25}
\end{subfigure}
\caption{The realizable tropical half lines to Example \ref{ex:trivsub}}\label{fig:halflines}
\end{figure}{}

	\begin{exa}\label{ex:1}
		The concluding example is a subdivision relatively close to a unimodular triangulation: it contains only two polytopes of the too large volume $\frac{1}{3}$. However, we can still observe lifting multiplicities higher than one, but we can no longer approximate lines partly contained in the infinite boundary by families.
		
		We choose the parameters $a = t^{20}$, $b= 1+t-t^6$, $c=-1-t+t^6$, $d=-1-t+t^6+ t^8$, $g= -1 -t+t^6+t^8+t^{15},$
		and obtain the weight vector $w=(0,8,6,15,8,20,0,20,0,0)$, which induces a subdivision with $8$ maximal cells; see Figure \ref{fig:ex1_subdiv}. The corresponding surface has 1 isolated line $L_{iso}$ and 2 families $F_0$ and $F_1$, computed using \texttt{a-tint} in \texttt{polymake}:
	\begin{align*}
	    F_0&=[0,0,8,0,8,12+\lambda],  \lambda\geq 0\\
	    F_1&=[0,0,8,0,8+\lambda,8], \lambda\geq 0\\
	    L_{iso}&= [0,15,15,0,0,20].
	\end{align*}

		The primal motifs of the two families can be seen in Figure \ref{fig:ex1_primalmotifs}. We obtain motifs that are not in the classification of \cite{PaVi19}, since our surface is not tropically smooth. 
		Only the motif of the isolated line, motif 3A, is one that can occur on general smooth tropical cubics.
		
		Computing the tropicalized lines with finite Pl\"ucker coordinates using Table \ref{tab:plückercoordinates}, we obtain that $\trop{F_{36}}=\trop{F_{56}}=L_{iso}$, while $E_4$ tropicalizes to an inner member of $F_0$ given by $\lambda=8$ and $F_{35}$ tropicalizes to the degenerate representative of $F_1$ with $\lambda=0$.
		In this case, we cannot obtain any of the other $23$ lines as limits of the families as $\lambda\to\infty$.	\end{exa}
		\begin{figure}[h]
			\centering
			\begin{subfigure}{.3\textwidth}
  \centering
 \includegraphics[width=0.5\linewidth]{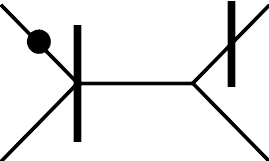} 
    \caption{Primal motif of $F_0$}
\end{subfigure} 
\begin{subfigure}{.3\textwidth}
  \centering
 \includegraphics[width=0.5\linewidth]{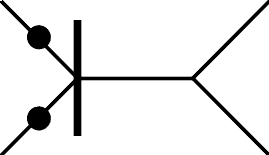} 
    \caption{Primal motif of $F_1$}
\end{subfigure} 
\begin{subfigure}{.3\textwidth}
  \centering
 \includegraphics[width=0.5\linewidth]{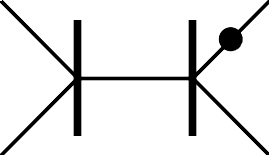} 
    \caption{Primal motif 3A of the isolated line}
\end{subfigure} 
			\caption{Primal motifs of the lines from Example \ref{ex:1}}\label{fig:ex1_primalmotifs}
		\end{figure}

	\begin{figure}[h]
	    \centering
	    \begin{subfigure}{.3\textwidth}
	    \includegraphics[width=.8\linewidth]{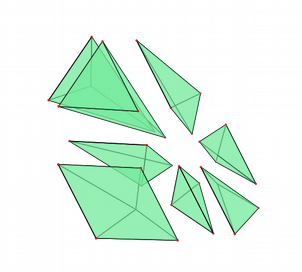}
	    \caption{Dual subdivision }
	    \label{fig:ex1_subdiv}
	\end{subfigure}{}
		\begin{subfigure}{.4\textwidth}
			\centering
			\includegraphics[width=0.7\linewidth]{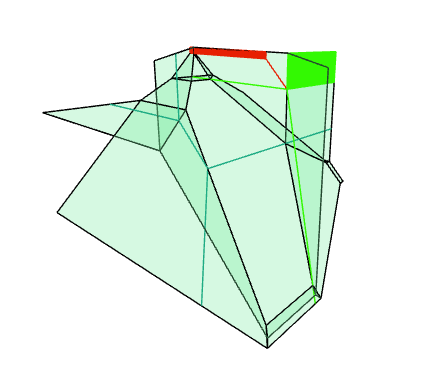}
			\caption{The tropicalized surface with its isolated line and two families.
				$F_0$ in red, $F_1$ in green.}\label{fig:all_lines_ex1}
				\end{subfigure}
				\caption{Dual subdivision and surface to Example \ref{ex:1}}
		\end{figure}

	\bibliography{Bibliographie}

\end{document}